\documentclass{amsart}

\usepackage{float}
\usepackage{subfiles}
\usepackage{amsmath,amsfonts,amscd,amssymb,enumitem,color}
\usepackage{float}

\usepackage{layout}
\usepackage{graphicx}
\usepackage{cite}
\usepackage{color}
\usepackage{appendix}
\usepackage{amssymb,latexsym,amsthm}
\usepackage{amsfonts}
\usepackage{mathtools}

\usepackage{multirow}

\usepackage{xcolor}

\usepackage{nomencl}
\makenomenclature

\usepackage{tikz-cd}
\usetikzlibrary{matrix, calc, arrows}
\usetikzlibrary{arrows.meta}

\def\newmathop#1{\expandafter\gdef\csname #1\endcsname{\mathop{\rm #1}\nolimits}}
\def\newvmathop#1{\expandafter\gdef\csname v#1\endcsname{\mathop{\rm #1}\nolimits}}

\theoremstyle{plain}
\newcounter{thmcount}[section]

\newtheorem{theorem}[thmcount]{Theorem}

\newtheorem{lemma}[thmcount]{Lemma}

\newtheorem*{conjecture*}{Conjecture}
\newtheorem*{theorem*}{Theorem}
\newtheorem*{corollary*}{Corollary}
\newtheorem*{lemma*}{Lemma}

\newtheorem*{rough*}{Rough idea}

\newtheorem*{motivation*}{Motivation}

\newtheorem*{goal*}{Goal}

\theoremstyle{definition}

\newtheorem{example}[thmcount]{Example}
\newtheorem*{example*}{Example}
\newtheorem*{remark*}{Remark}
\newtheorem{definition}[thmcount]{Definition}

\numberwithin{equation}{section}

\def\Q{{\mathbb Q}}

\def\Fpb{{\overline{\mathbb F}_p}}

\def\Qp{{{\mathbb Q}_p}}

\def\Qpb{{\overline{\mathbb Q}_p}}
\def\Z{{\mathbb Z}}

\def\C{{\mathbb C}}

\newmathop{Disc}
\newmathop{Tr}
\newmathop{Norm}
\newmathop{ord}
\newmathop{GL}
\newmathop{SL}
\newmathop{PGL}
\newmathop{Hom}
\newmathop{Ind}
\newmathop{Res}
\newmathop{rk}
\newmathop{corank}
\newmathop{coker}
\newmathop{codim}
\newmathop{cyc}
\newmathop{Reg}
\newmathop{Gal}
\newmathop{Sel}

\newmathop{BSD}
\newmathop{tors}
\newmathop{tr}
\newmathop{Ext}
\newmathop{Mod}

\newmathop{id}
\newmathop{Aut}
\newmathop{Art}

\newmathop{Rep}

\DeclareMathOperator{\Ha}{Ha}

\renewcommand{\det}{\operatorname{det}}

\newcommand{\mtwo}[4]{\begin{pmatrix} #1 & #2  \\ #3 & #4 \end{pmatrix}}
\newcommand{\mtwosmall}[4]{\begin{psmallmatrix} #1 & #2  \\ #3 & #4 \end{psmallmatrix}}
\DeclareMathOperator{\Frob}{Fr}

\newcommand{\tta}[1]{\hat{#1}}

\begin{document}

\title{Geometric modularity for algebraic and non-algebraic weights}

\author{Hanneke Wiersema}
\date{\today}
\email{hw600@cam.ac.uk}
\address{University of Cambridge, Department of Pure Mathematics and Mathematical Statistics, Centre for Mathematical Sciences, Wilberforce Road, Cambridge CB3 0WB}

\begin{abstract}
In this short paper we generalise a result of Diamond--Sasaki connecting geometric modularity of algebraic weights to geometric modularity of non-algebraic weights and vice versa. In particular, we show that geometric modularity of non-algebraic weights implies geometric modularity of multiple algebraic weights, which is easy to see. More significantly, we show that geometric modularity of multiple algebraic weights sometimes implies geometric modularity of a non-algebraic weight. This connects work of Diamond--Sasaki to earlier work on generalisations of the weight part of Serre's conjecture.
\end{abstract}
\maketitle

Serre's modularity conjecture \cite{serreduke}, now a theorem of Khare and Wintenberger, asserts a connection between two-dimensional mod $p$ Galois representations of the rationals and modular forms. The strong version includes a prediction for the minimal weight $k \geq 2$ such that a modular form of weight $k$ corresponds to a given Galois representation, this is the \emph{weight part} of Serre's conjecture. Edixhoven refined the weight part by including modular forms of weight one by considering {geometric} mod $p$ modular forms in \cite{Edixhoven1992}. 
\par
More generally, it is conjectured that such a connection should also exist between mod $p$ representations of $G_F$, where $F$ is a totally real field, and Hilbert modular forms. In this setting, Diamond and Sasaki give an analogue of Edixhoven's version of the weight part of Serre's conjecture by considering geometric mod $p$ Hilbert modular forms in \cite{DS}, based on work by Andreatta and Goren \cite{andreattagoren}. Crucially, they attach Galois representations to mod $p$ Hilbert eigenforms of \emph{arbitrary} weights, extending work by Goldring and Koskivirta \cite{goldringkoskivirta} and Emerton, Reduzzi and Xiao \cite{emertonreduzzixiao} for paritious weights. 
\par
In the classical case, the geometric approach allowed for the inclusion of weight one modular forms. In the more general setting we obtain many more new weights: those of partial weight one. We wish to connect the partial weight one setting to previous work on generalisations of Serre's conjecture where the weights are algebraic (i.e. each integer in the tuple of weights is at least two), most notably the work of Buzzard, Diamond and Jarvis \cite{bdj}, and of Gee et al. \cite{geekisin},\cite{GLS14},\cite{GLS15}. In the case where $F$ is a real quadratic field in which $p$ is inert, Diamond and Sasaki proved a connection between modularity of algebraic weights and modularity of non-algebraic weights \cite[Lemma 11.3.1]{DS}. 
\par
The main theorem in this paper (Theorem \ref{mainthm}) relates modularity of non-algebraic weights to modularity of algebraic weights for arbitrary totally real $F$ in which $p$ is unramified. As we shall see, the work of generalising the result of Diamond--Sasaki mostly lies in dealing with the more complicated combinatorics and in finding explicit formulas for the weights.
\par
Therefore key to this work is, given any non-algebraic weight, finding a set of corresponding algebraic weights whose modularity will imply modularity of the non-algebraic weight (under some technical conditions). We describe this process in Section \ref{weightcombinatorics}, after first introducing the modular forms involved in Section \ref{background}. In Section \ref{statement} we give the precise statement of our main result as well as the proof. The proof uses methods introduced in \cite{DS} involving partial Hasse invariants and partial Theta operators. 
\par
The results in this paper complement the work on the parallel weight one case by Dimitrov--Wiese \cite{dimitrovwiese} and Gee--Kassaei \cite{geekassaei}. The former proved that the Galois representation attached to a mod $p$ Hilbert modular eigenform of parallel weight 1 and level prime to $p$ is unramified above $p$, and the latter proved, under some mild conditions, that if a modular representation $\rho$ is unramified at every $v \mid p$, that then there exists a form $f$ of parallel weight 1 such that $\rho \sim \rho_f$.

\subsection*{Notation}\label{ss:notation}
We let $p$ be a prime number. Fix algebraic closures $\Qpb$ and $\overline{\Q}$ of $\Qp$ and $\Q$ respectively, and fix embeddings of $\overline{\Q}$ into $\Qpb$ and $\C$. Let $F$ be a totally real field in which $p$ is unramified.  We let $\mathcal{O}_F$ be the ring of integers of $F$ and $\mathcal{O}_{F,p}=\mathcal{O}_F \otimes \Z_p$. We let $\hat{\mathcal{O}}_F$ denote the profinite completion of $\mathcal{O}_F$. Let $\Sigma$ be the set of embeddings $F \to \overline{\Q}$.

We let $L/\Q_p$ be a sufficiently large finite extension of $\Qp$ containing the image of every embedding in $\Sigma$. Let $E$ be the residue field. We identify $\Sigma$ with the set of embeddings of $F$ into $L$ and with the set of embeddings of $\mathcal{O}_F/p$ into $\Fpb$. We let $\Frob$ be the absolute Frobenius of $\Fpb$. For $v \mid p$ we denote $\Sigma_v$ for the set of embeddings of $\mathcal{O}_F/v$ into $\Fpb$. Note that Frobenius has a cyclic action on $\Sigma_v$. For any $\tau_0 \in \Sigma_v$ we let $\tau_i=\tau_{i+1}^p$ so that $\tau_i= \Frob \circ \tau_{i+1}$.  We often identify $\Sigma$ with $\prod_{v \mid p} \Sigma_v$.

For any $\tau \in \Sigma$, we let $e_{\tau}$ denote the basis element of $\Z^{\Sigma}$ associated to $\tau$.

\subsection*{Acknowledgements}
I would like to thank Fred Diamond for suggesting the problem and for helpful discussions. I would like to thank both Fred Diamond and James Newton for comments on earlier versions of this paper. As the reader will notice, the methods of this paper owe a debt to the work of Diamond--Sasaki and it is a pleasure to thank them both for conversations about their work.

\section{Background}
\label{background}
\subsection{Mod $p$ Hilbert modular forms}
In this paper mod $p$ Hilbert modular forms will be sections of automorphic line bundles on Hilbert modular varieties of level prime to $p$ in characteristic $p$ as defined in \cite[Def 3.2.2]{DS}. We let $M_{k,l}(U;E)$ be the space consisting of all mod $p$ Hilbert modular forms $f$ of weight $(k,l) \in \Z^\Sigma \times \Z^\Sigma$ and of level $U$, where $U$ is any open compact subgroup of $\GL_2(\hat{\mathcal{O}}_F)$ containing $\GL_2(\mathcal{O}_{F,p})$, and with coefficients in $E$. Recall we say a weight $(k,l)$ is \emph{algebraic} if $k \in \Z_{\geq 2}^{\Sigma}$.

We will need that the minimal weight of a mod $p$ Hilbert modular form as in \cite[Section 5.2]{DS}  lies in a minimal cone $\Xi_{\min}$ \cite[Corollary 1.2]{DiamondKassaei}:
\[
\Xi_{\min} = \{ k \in \Z^{\Sigma} \, \vert \, pk_{\tau} \geq k_{\Frob^{-1} \circ \tau} \text{ for all } \tau \in \Sigma \}.
\]

If $\mathfrak{n}$ is a prime ideal of $\mathcal{O}_F$ we define $U_1(\mathfrak{n})$ as the open compact subgroup of $\GL_2(\hat{\mathcal{O}}_F)$ given by
\[
U_1(\mathfrak{n})=\left\{ \mtwo{a}{b}{c}{d} \in \GL_2(\hat{\mathcal{O}}_F) \, \vert \, c, d-1 \in \mathfrak{n} \hat{\mathcal{O}}_F \right\}.
\]

In the following, we write $S_v$ and $T_v$ for the Hecke operators defined in \cite[Section 4.3]{DS} (for $v \nmid p$ and arbitrary weight $(k,l)$) and \cite[Section 9.7]{DS} (for $v \mid p$ and for algebraic weight $(k,0)$).

\subsection{Partial Hasse invariants and partial $\Theta$ operators}
\label{partoperators}
For any $\tau \in \Sigma$, we let $\Ha_{\tau}$ denote the partial Hasse invariant defined in \cite[Section 5.1]{DS}. We write $k_{\Ha_{\tau}}= p e_{\Frob^{-1} \circ \tau} - e_{\tau}$ for the weight of $\Ha_{\tau}$. Multiplication by $\Ha_{\tau}$ gives a new mod $p$ Hilbert modular form $\Ha_{\tau}f$ if $f$ has level $U$ where $U$ is a subgroup of $\GL_2(\mathbb{A}_F^{\infty})$ containing $\GL_2(\mathcal{O}_{F,p})$; this is Proposition 5.1.1 of \cite{DS}. In particular, we have an injective map
\[
M_{k,l}(U;E) \to M_{k+k_{\Ha_{\tau}},l}(U;E),
\]
commuting with $T_v$ and $S_v$ for all $v \nmid p$ such that $\GL_2(\mathcal{O}_{F,v}) \subset U$.

For any $\tau \in \Sigma$, we let $\Theta_{\tau}$ denote the partial $\Theta$ operator defined in \cite[Definition 8.2.1]{DS}. By Theorem 10.4.2 of \cite{DS}, for any $f \in M_{k,l}(U;E)$ and $\tau \in \Sigma$ we obtain a new Hilbert modular form $\Theta_{\tau}(f) \in M_{\tta{k},\tta{l}}(U;E)$ where
\begin{itemize}
\item if $\Frob \circ \tau=\tau$, then $\tta{k}_{\tau}=k_{\tau}+p+1$ and $\tta{k}_{\tau'}=k_{\tau'}$ if $\tau' \neq \tau$;
\item if $\Frob \circ \tau \neq \tau$, then $\tta{k}_{\tau}=k_{\tau}+1, \tta{k}_{\Frob^{-1} \circ \tau}=k_{\Frob^{-1} \circ \tau}+p$ and $\tta{k}_{\tau'}=k_{\tau'}$ if $\tau \not \in \{ \Frob^{-1} \circ \tau, \tau\}$;
\item $\tta{l}_{\tau}=l_{\tau}-1$, and $\tta{l}_{\tau'}=l_{\tau'}$ if $\tau' \neq \tau$.
\end{itemize}

\subsection{Properties of Hilbert modular eigenforms}
In this section we assume $(k,l)$ is algebraic. We introduce some properties of eigenforms, details of each can be found in \cite[Section 10]{DS}.

\subsubsection{Twisted forms}
Let $\mathfrak{m}$ be an ideal of $\mathcal{O}_F$ such that $\mathfrak{m} \mid \mathfrak{n}$. We let $V_{\mathfrak{m}} \subset \hat{\mathcal{O}}^{\times}_F$ denote the kernel of the natural projection to $(\mathcal{O}_F/\mathfrak{m})^{\times}$. We say a character
\[
\xi: \{ a \in (\mathbb{A}_F^{\infty})^{\times} \mid a_p \in \mathcal{O}_{F,p}^{\times}/V_{\mathfrak{m}}\} \to E^{\times},
\]
is a character of weight $l' \in \Z^{\Sigma}$ if $\xi(\alpha)=\overline{\alpha}^{l'}$ for all $\alpha \in F_{+}^{\times} \cap \mathcal{O}_{F,p}^{\times}$.

For any $f \in M_{k,l}(U_1(\mathfrak{n});E)$ and any character of weight $l'$ and conductor $\mathfrak{m}$, we write $f_{\xi} \in M_{k,l+l'}(U_1(\mathfrak{n}\mathfrak{m}^2);E)$ for the twisted form defined in \cite[Section 10.3]{DS}.

\subsubsection{Normalised eigenforms}
A \emph{normalised eigenform} is a Hilbert modular form $f$ which is an eigenform for Hecke operators $T_v$ for all $v \nmid p$ and $S_v$ for all $v \nmid \mathfrak{n}p$, and where all $f_{\xi}$, twists of $f$ by characters $\xi$ of weight $-l$ and conductor $\mathfrak{m}$ coprime to $p$, with values in extensions $E'$ of $E$, are eigenforms for $T_v$ for all $v \mid p$. Finally, the coefficient of $q$ in the $q$-expansion should be $1$. If $f$ is a normalised eigenform in $M_{k,l}(U_1(\mathfrak{n}),E)$ and $\xi$ is a character of weight $-l'$ and conductor $\mathfrak{m}$ coprime to $p$, then $f_{\xi}$ is a normalised eigenform in $M_{k,l+l'}(U_1(\mathfrak{n}\mathfrak{m}^2);E)$.

\subsubsection{Stabilised eigenform}
A normalised eigenform $f \in M_{k,l}(U_1(\mathfrak{n});E)$ is \emph{stabilised} if $T_vf=0$ for all $v \mid \mathfrak{n}$. A stabilised eigenform is \emph{strongly stabilised} if and only if $T_v f_{\xi}=0$ for all $v \mid p$ and characters $\xi$ of weight $-l$.

Let us note that this property is crucial to the proof of the main result. We need that there is at most one strongly stabilised eigenform of a given level and weight giving rise to a given Galois representation, as we will see.

If $f$ is a normalised (resp. stabilised, strongly stabilised) eigenform, then the same is true for both $\Ha_{\tau}f$ and $\Theta_{\tau}(f)$ for any $\tau$ (assuming $k_{\tau} \geq 3$ if $\tau \neq \Frob \circ \tau$ in the case of $\Ha_{\tau}f$). This is \cite[Remark 10.6.6]{DS}.

\subsection{Geometric modularity}
Let $U$ be an open compact subgroup $U \subset \GL_2(\hat{\mathcal{O}}_F)$ containing $\GL_2(\mathcal{O}_{F,p})$. Let $Q$ be a finite set of primes containing all $v \mid p$ and all $v$ such that $\GL_2(\mathcal{O}_{F,v}) \not \subset U$. Suppose $f \in M_{k,l}(U;E)$ is an eigenform for $T_v$ and $S_v$ for all $v \not \in Q$, then we write $\rho_f$ for the Galois representation attached to $f$ as in \cite[Theorem 6.1.1]{DS}.

\begin{definition}
An irreducible, continuous and totally odd Galois representation $\rho: G_F \to \GL_2(\Fpb)$ is \emph{geometrically modular of weight $(k,l)$} if $\rho$ is equivalent to the extension of scalars of $\rho_f$ for some eigenform $f \in M_{k,l}(U;E)$ as above. 
\end{definition}

This means $\rho$ is geometrically modular of weight $(k,l)$ if there is a non-zero element $f \in M_{k,l}(U;E)$ for some $U \supset \GL_2(\mathcal{O}_{F,p})$ and $E \subset \Fpb$ such that
\[
T_v f=\tr(\rho(\Frob_v))f \quad \text{ and } \text{Nm}_{F/\Q}(v) S_v f = \det(\rho(\Frob_v))f,
\]
for all but finitely many primes $v$.

In the following theorem we collect some of the results we need for our proof, which are results from \cite{DS}.

\begin{theorem}
\label{DSmain}
Suppose $\rho: G_F \to \GL_2(\Fpb)$ is irreducible and geometrically modular of weight $(k,l)$.
\begin{enumerate}
\item Then $\rho$ arises from an eigenform $f$ of weight $(k,l)$ and level $U_1(\mathfrak{n})$ for some $\mathfrak{n}$ prime to $p$.
\item Moreover, we can take $f$ such that $\Theta_{\tau}(f) \neq 0$ for all $\tau \in \Sigma$.
\item If $k_{\tau} \geq 2$ for all $\tau$, then $f$ can moreover be taken to be a normalised eigenform of level $U_1(\mathfrak{n})$ for some $\mathfrak{n}$ prime to $p$.
\item Moreover, $\rho$ is also geometrically modular of weight $(k+k_{\Ha_{\tau}},l)$ for any $\tau \in \Sigma$.
\item Moreover, $\rho$ is also geometrically modular of weight $(\tta{k},\tta{l})$ as in Section \ref{partoperators} for any $\tau \in \Sigma$.
\end{enumerate}
\end{theorem}

\begin{proof}
\begin{enumerate}
\item This is \cite[Lemma 10.2.1]{DS}.
\item This is \cite[Lemma 10.4.1]{DS}.
\item This is \cite[Proposition 10.5.2]{DS}.
\item This follows from the commutativity of $\Ha_{\tau}$ with the Hecke operators.
\item This is \cite[Theorem 10.4.2]{DS}, which follows from $(2)$ and commutativity of $\Theta_{\tau}$ with the Hecke operators. \qedhere
\end{enumerate}
\end{proof}

\section{Weight combinatorics}
\label{weightcombinatorics}
In this section we prepare to state the main result connecting modularity of algebraic weights to non-algebraic weights and vice versa. We introduce the Hilbert modular forms with algebraic weights we associate to Hilbert modular forms of some non-algebraic weight.

\subsection{The algebraic weights}
\label{algweights}
Suppose $f$ is a non-zero mod $p$ Hilbert modular form with non-algebraic weight $(k,0)$ such that $k_{\tau}>0$ for all $\tau \in \Sigma$. We associate multiple Hilbert modular forms to it, all with algebraic weights. The first form is obtained from multiplication by Hasse invariants, and the other ones from multiplication by Hasse invariants and then acting on it with a partial Theta operator.

\begin{definition}
We let $M$ be the subset of $\Sigma$ consisting of $\tau \in \Sigma$ such that
\[
 k_{\Frob^{-1} \circ \tau}=\dots =k_{\Frob^{(1-s)} \circ \tau}=2 \text{ and } k_{\Frob^{-s} \circ \tau}=1,
\]
for some $s \geq 1$ (the first condition is vacuous if $s=1$).
\end{definition}

\subsubsection*{The definition of $k'$}
We let $f'$ be the modular form 
\begin{align}
\label{defkprime}
f \cdot \prod_{\tau \in M} \Ha_{\tau}. 
\end{align}
We denote $(k',l')$ for the weight of $f'$.  We have $l'=l=(0, \dots,0)$. As the weight of $\Ha_{\tau}$ is $p e_{\Frob^{-1} \circ \tau} - e_{\tau}$ the weight of $f'$ can easily be seen to be algebraic.

\begin{example}
Say $k=(1,1,k_2)$. If $k_2 \geq 3$ we multiply $f$ by $\Ha_2$ and $\Ha_0$ to obtain $f'$ of weight $(k',0)$ with $k'=(p,p+1,k_2-1)$. If $k_2=2$, then we multiply $f$ by $\Ha_2$, $\Ha_0$ and $\Ha_1$ to obtain $f'$ of weight $(k',0)$ with $k'=(p,p,p+1)$.
\end{example}

\subsubsection*{The definition of $k^{\mu}$}

We define several modular forms $f^{\mu}$ obtained from $f$ by using partial Hasse invariants and partial Theta operators. 

\begin{definition}
Let $\tilde{M}$ be the set of $\tau \in \Sigma$ such that one of the following conditions holds:
\begin{itemize}
\item  $k_{\tau} \geq 3$ and $\tau \in M$,
\item $k_{\tau}=2, k_{\Frob^{-1} \circ \tau}=1, k_{\Frob^{-2} \circ \tau}=1$ or $2$ and if $k_{\Frob^{-2} \circ \tau}=2$ then $\Frob^{-2} \circ \tau \in M$.
\end{itemize}
\end{definition}

Note that if there is no case where $k_{\tau}=2$ and $k_{\Frob^{-1} \circ \tau}=1$, that $\tilde{M}$ is just the set $\tau \in \Sigma$ such that $k_{\tau} \neq 1, k_{\Frob^{-1} \circ \tau}=1$.

Let $\mu \in \tilde{M}$, then we define $f^{\mu}$ to be the form given by
\begin{align}
\label{defkmu}
\Theta_\mu \left( f \cdot \prod_{\tau \in {M} \setminus \{\mu\}} \Ha_{\tau} \right).
\end{align}
We denote $(k^{\mu},l^{\mu})$ for the weight of $f^{\mu}$.  We have $l^{\mu}= - e^{\mu}$. The Theta operator adds $p e_{\Frob^{-1} \circ \tau} + e_{\tau}$ to the weight of the form it acts on. It can easily be verified that the weight of $f^{\mu}$ is algebraic. 

\begin{example}
Say $k=(1,1,k_2)$ with $k_2 \geq 3$, then we multiply $f$ by $\Ha_0$ and act on this with $\Theta_2$ to obtain $f^{\mu}$ of weight $(k^{\mu},l^{\mu})=((p,p+1,k_2+1),(0,0,-1))$.
If $k_2 =2$, then we multiply $f$ by $\Ha_0$ and $\Ha_1$, act on this with $\Theta_2$ to obtain $f^{\mu}$ of weight $(k^{\mu},l^{\mu})=((p,p,p+3),(0,0,-1))$.
\end{example}

\section{Statement and proof of the theorem}
\label{statement}

\begin{theorem}
\label{mainthm}
Let $G_F \to \GL_2(\Fpb)$ be an irreducible representation. Suppose $\rho$ is geometrically modular of a non-zero non-algebraic weight $(k,0)$ such that $k \in \Xi_{\min}$ and $k_{\tau}>0$ for all $\tau \in \Sigma$. Then 
\begin{enumerate}
\item $\rho$ is geometrically modular of weight $(k',0)$,
\item for each $\mu \in \tilde{M},$ $\rho$ is geometrically modular of weight $(k^{\mu},l^{\mu})$.
\end{enumerate}
Moroever, the converse holds if we assume in addition that
\begin{enumerate}[resume]
\item for $\mu \in \tilde{M}$ let $f^{\mu}$ be the form as defined above. For each $\mu \in \tilde{M}$ we have $T_v((f^{\mu})_{\xi})=0$ for all $v \mid p$ and characters $\xi$ of weight $-l^{\mu}$,
\item there is no place $v \mid p$ such that $k_{\tau}=1$ for all $\tau \in \Sigma_v$, 
\item $k_{\mu} \not \equiv 1 \mod p$ for all $\mu \in \tilde{M}.$
%\item $p$ is inert in $F$ and for the prime $v|p$ the local representation $\rho|_{G_{F_v}}$ is not of the form $\mtwosmall{\chi_1}{\ast}{0}{\chi_2}$ where $\chi_1|_{I_{F_v}}=\omega_{\mu}$ for any $\mu \in \tilde{M}$, 
\end{enumerate}
\end{theorem}

We note first that the theorem holds just as well for $(k,l)$ (with obvious adjustments to Section \ref{algweights}), we take $l=0$ to ease exposition.

The condition (3) is implied by demanding that for all primes $v|p$ the local representation $\rho|_{G_{F_v}}$ is not of the form $\mtwosmall{\chi_1}{\ast}{0}{\chi_2}$ where $\chi_1|_{I_{F_v}}=\omega_{\mu}$ for any $\mu \in \tilde{M}$ and applying \cite[Corollary 10.7.2]{DS}. The reason for this condition is to ensure we obtain strongly stabilised forms which we will need to show divisibility by certain Hasse invariants. The key to the proof is that there is at most one strongly stabilised eigenform $f  \in M_{k,l}(U_1(\mathfrak{n};E))$ giving rise to $\rho$.

We exclude the weights in (4) as these would require methods similar to the aforementioned \cite{dimitrovwiese} and \cite{geekassaei}. The weights in (5) are excluded as our methods break down for that case.

\begin{proof}
First assume $\rho$ is geometrically modular of $f \in M_{(k,0)}(U;E)$. By Theorem \ref{DSmain}(4) it is also geometrically modular of weight $k'$ by multiplying $f$ by the Hasse invariants as in (\ref{defkprime}). Now for each $f^{\mu}$ we proceed similarly, we multiply by Hasse invariants and act on the result by partial Theta operators as in (\ref{defkmu}) in which case it follows from Theorem \ref{DSmain}(4-5).
\par
For the converse, suppose we have  (1)-(5). 
By (1) and Theorem \ref{DSmain}(3) $\rho$ arises from a normalised eigenform in $M_{{k'},{0}}(U_1(\mathfrak{m});E)$ for an ideal $\mathfrak{m}$ prime to $p$ (and sufficiently large $E$). 

From (2), and for any $\mu \in \tilde{M}$, Theorem \ref{DSmain}(3) again implies that $\rho$ arises from a normalised eigenform in $M_{{k}^{\mu},{l}^{\mu}}(U_1(\mathfrak{m}_{\mu});E)$ for some ideal $\mathfrak{m}_{\mu}$ prime to $p$.

Let $\mathfrak{n}=\mathfrak{m}^2  \prod_{\mu \in \tilde{M}} \mathfrak{m}_{\mu}^2$. Then $\mathfrak{n} \subset \mathfrak{m}$ satisfies the conditions in \cite[Lemma 10.6.2]{DS}, so this means $\rho$ arises from a stabilised eigenform $f' \in M_{k',0}(U_1(\mathfrak{n});E)$. Similarly, for each $\mu \in \tilde{M}$, the ideal $\mathfrak{n} \subset \mathfrak{m}_{\mu}$  satisfies the conditions in \cite[Lemma 10.6.2]{DS} with $\mathfrak{m}=\mathfrak{m}_{\mu}$, so this means $\rho$ arises from a stabilised eigenform $f^{\mu} \in M_{{k}^{\mu},{l}^{\mu}}(U_1(\mathfrak{n});E)$.

Fix $\mu \in \tilde{M}$. By \cite[Proposition 9.4.1]{DS} $\Theta_{\mu} (f')$ will be a strongly stabilised eigenform of weight $(k'+p e_{\mu} + e_{\Frob^{-1} \circ \mu}, - e_{\mu})$ and level $\mathfrak{n}$.

By (3) $\Ha_{\mu}(f^{\mu})$ is a strongly stabilised eigenform of weight $(k^{\mu}+p e_{\mu} - e_{\Frob^{-1} \circ \mu}, - e_{\mu})$ and level $\mathfrak{n}$. 

Now by \cite[Lemma 10.6.5]{DS}, the modular form that $\rho$ arises from with these specifics is unique so that
\[
\Theta_{\mu}(f')=\Ha_{\mu}(f^{\mu}).
\]
By Theorem 8.2.2 of \cite{DS} and (5) we know that $f'$ is divisible by $\Ha_{\mu}$.

We can do this step for every $\mu \in \tilde{M}$ to find that $\rho$ is geometrically modular of a form $g$ such that
\begin{align}
\label{kg}
f'=g \prod_{\mu \in \tilde{M}} \Ha_{\mu}.
\end{align}
The form $g$ is of level $\mathfrak{n}$ and of weight $k' - \sum_{\mu \in \tilde{M}} k_{\Ha_{\mu}}$. If this equals $k$, the proof is done. 

If not, this means that the set $M':=M \setminus \tilde{M}$ is non-empty. We finish the proof by showing $g$ is divisible by $\prod_{\tau \in {M'}} \Ha_{\tau}$. Write $k''$ for the weight of $g$. Suppose $\tau \in M'$, then $\tau$ falls into one of the following three cases:
\begin{enumerate}[label=\roman*]
\item $k_{\tau}=1, k_{\Frob^{-1} \circ \tau}=1$,
\item $k_{\tau}=2, k_{\Frob^{-1} \circ \tau}=1, k_{\Frob^{-2} \circ \tau} \neq 1$, and if $k_{\Frob^{-2} \circ \tau}=2$ then $\Frob^{-2} \circ \tau \not \in M$,
\item $k_{\tau}=2, k_{\Frob^{-1} \circ \tau}=2$ with $\tau \in M$.
\end{enumerate}

In the first case, for some $s \geq 1$ there must be an embedding $\Frob^{s} \circ \tau \in \tilde{M}$ such that $k_{\Frob^{s-1} \circ \tau}= \dots = k_{\tau}=1$. We find $(k''_{\Frob^{s-1} \circ \tau}, \dots, k''_{\tau}, k''_{\Frob^{-1} \circ \tau})=(0,p,p, \dots, p)$ or $(0,p, \dots,p, p+1)$. Theorem 1.1 of \cite{DiamondKassaei}, which implies divisibility by $\Ha_{\tau}$ if $p k_{\tau}<k_{\Frob^{-1} \circ \tau}$, gives us divisibility of $\Ha_{\Frob^{s-1} \circ \tau}$. The result now follows from Lemma \ref{HasDiv}.

In the second case, there exist $s,t \geq 0$ such that $\Frob^{s+t} \circ \tau \in \tilde{M}$ and
\begin{align}
\label{case2}
k_{\Frob^{s+t-1} \circ \tau}=\dots=k_{\Frob^s \circ \tau}=1, k_{\Frob^{s-1} \circ \tau}= \dots = k_{\tau}=2,
\end{align}
If $t=0$, we have $(k''_{\Frob^{s+t-1} \circ \tau}, \dots, k''_{\tau}, k''_{\Frob^{-1} \circ \tau})=(1,p+1,p+1, \dots, p+1)$. By \cite[Theorem 1.1]{DiamondKassaei} and Lemma \ref{HasDiv} we obtain divisibility by $\Ha_{\Frob^{s+t-1} \circ \tau} \cdots \Ha_{\tau}$. If $t \geq 1$, then we obtain $(0,p, \dots,p, p+1,\dots, p+1)$. We get divisibility by $\Ha_{\Frob^{s+t-1} \circ \tau} \cdots \Ha_{\Frob^s \circ \tau}$ by \cite[Theorem 1.1]{DiamondKassaei} and Lemma \ref{HasDiv}. We obtain divisibility by $\Ha_{\Frob^{s-1} \circ \tau} \cdots \Ha_{\tau}$ from again applying Lemma \ref{HasDiv}.  The third case is covered in the second case (note here $\tau \in M$ implies $\Frob^{-1} \circ \tau \in M$). 
\end{proof}

\begin{lemma}
\label{HasDiv}
Suppose $f$ is of weight $(k,l)$ with $k \in \Z^{\Sigma}_{\geq 0}$. Suppose we know $f$ is divisible by $\Ha_{\tau}$ and suppose for some $s \geq 1$ we have 
\[
k_{\Frob^{-1} \circ \tau}= \dots= k_{\Frob^{-s} \circ \tau}=m, k_{\Frob^{-s-1}}=m+1,
\]
for some positive integer $m \leq p+1$. Then $f$ is divisible by $\Ha_{\tau} \cdots \Ha_{\Frob^{-s} \circ \tau}$.
\end{lemma}

\begin{proof}
We use Theorem 1.1 of \cite{DiamondKassaei}, which implies divisibility by $\Ha_{\tau}$ if $p k_{\tau}<k_{\Frob^{-1} \circ \tau}$. After dividing by $\Ha_{\tau}$ the entry corresponding to  $\Frob^{-1} \circ \tau$ is $m-p$. We have $p(m-p)<m$ since $m \leq p+1$. It is easy to see that $p(k_{\Frob^{-(j-1)} \circ \tau}-p)<k_{{\Frob^{-j} \circ \tau}}$ for all $j \in \{2,\dots, s+1\}$. This concludes the proof.
\end{proof}

\begin{example}
Assume $p$ is inert in $F$ and suppose $k=(1,1,k_2,2,2,1,2,2)$ with $k_2 \geq 3$. We have 
\begin{align*}
k'=&(p,p+1,k_2-1,p+1,p+1,p,p+1,p+1), \\
k^{\mu_1}=&(p,p+1,k_2+1,p+1,p+1,p,p+1,p+1), \\
k^{\mu_2}=&(p,p+1,k_2-1,p+1,p+3,p,p+1,p+1), \\
k^{\mu_3}=&(p,p+1,k_2-1,p+1,p+1,p,p+1,p+3). 
\end{align*}
By the above matching procedure we find $k''=(0,p+1,k_2,1,p+2,0,p+1,p+2)$,  the result follows from Theorem 1.1 of \cite{DiamondKassaei} and Lemma \ref{HasDiv}.
\end{example}

\bibliographystyle{alpha}

\bibliography{references}

\end{document}